\documentclass[11pt]{amsart}

\usepackage{amssymb}
\usepackage[cmtip,all]{xy}
\usepackage{amsmath, amsthm, enumerate}

\theoremstyle{plain}
\newtheorem{theorem}{Theorem}[section]
\newtheorem{lemma}[theorem]{Lemma}
\newtheorem{prop}[theorem]{Proposition}

\def\Hom{\operatorname{Hom}}
\def\Ind{\operatorname{Ind}}

\def\Isom{\operatorname{Isom}}

\def\GL{\mathrm{GL}}
\def\Mp{\mathrm{Mp}}

\def\Sp{\mathrm{Sp}}
\def\U{\mathrm{U}}

\def\CC{\mathbb{C}}

\begin{document}

\title[On the Howe duality conjecture]
{On the Howe duality conjecture\\in classical theta correspondence}
\author{Wee Teck Gan}
\address{Department of Mathematics, National University of Singapore, 10 Lower Kent Ridge Road, Singapore 119076}
\email{matgwt@nus.edu.sg}
\thanks{W.T. Gan is partially supported by an MOE Tier 1 Grant
  R-146-000-155-112  and an MOE Tier Two grant  MOE2012-T2-2-042.}

\author{Shuichiro Takeda}
\address{Mathematics Department, University of Missouri, Columbia, 202
  Math Sciences Building, Columbia, MO, 65211}
\email{takedas@missouri.edu} 
\thanks{S. Takeda is partially supported by NSF grant DMS-1215419.}
\dedicatory{to Jim Cogdell  \\ on the occasion of his 60th birthday}

\subjclass[2000]{Primary 11F27, Secondary 22E50}

\date{}

\maketitle

\begin{abstract}
We give a proof of the Howe duality conjecture for the (almost) equal
rank dual pairs. For arbitrary dual pairs, we prove
the irreducibility of the (small) theta lifts for all
tempered representations. Our proof works for any
nonarchimedean local field of characteristic not $2$ and in arbitrary residual characteristic.

\end{abstract}

\section{\textbf{Introduction}}

Let $F$ be a nonarchimedean local field of characteristic not $2$ and
residue characteristic $p$.  Let $E$ be $F$ itself or a quadratic
field extension of $F$.  For $\epsilon = \pm$, we consider a
$-\epsilon$-Hermitian space $W$ over $E$ of dimension $n$ and an
$\epsilon$-Hermitian space $V$ of dimension $m$.  We shall write
$W_n$ or $V_m$ if there is a need to be specific about the dimension
of the space in question.  Set
\[ \epsilon_0 = \begin{cases} \epsilon \text{ if $E=F$;} \\ 0 \text{
if $E \ne F$.}  \end{cases} \] 

\vskip 5pt

Let $G(W)$ and $H(V)$ denote the isometry group of $W$ and $V$
respectively.  Then the group $G(W) \times H(V)$ forms a dual
reductive pair and possesses a Weil representation $\omega_{\psi}$
which depends on a nontrivial additive character $\psi$ of $F$ (and
some other auxiliary data which we shall suppress for now).  To be
precise, when $E=F$ and one of the spaces, say $V$, is odd
dimensional, one needs to consider the metaplectic double cover of
$G(W)$; we shall simply denote this double cover by $G(W)$ as
well. The various cases are tabulated in \cite[\S 3]{gi}.  

\vskip 5pt

In the theory of local theta correspondence, one is interested in the
decomposition of $\omega_{\psi}$ into irreducible representations of
$G(W) \times H(V)$.  More precisely, for any irreducible admissible
representation $\pi$ of $G(W)$, one may consider the maximal
$\pi$-isotopic quotient of $\omega_{\psi}$. This has the form $\pi
\otimes \Theta_{W,V,\psi}(\pi)$ for some smooth representation
$\Theta_{W,V, \psi}(\pi)$ of $H(V)$; we shall frequently suppress $(W,
V,\psi)$ from the notation if there is no cause for confusion.  It was
shown by Kudla \cite{k83} that $\Theta(\pi)$ has finite length
(possibly zero), so we may consider its maximal semisimple quotient
$\theta(\pi)$.  One has the following fundamental conjecture due to
Howe \cite{howe}: 

\vskip 5pt

\noindent\underline{{\bf Howe Duality Conjecture for $G(W) \times
H(V)$}} \vskip 5pt
\noindent (i) $\theta(\pi)$ is either $0$ or irreducible.  \vskip 5pt

\noindent (ii) If $\theta(\pi) = \theta(\pi') \ne 0$, then $\pi =
\pi'$.  \vskip 10pt
 
We take note of the following theorem: \vskip 5pt

\begin{theorem}\label{T:kudla-walds}
\noindent (i) If $\pi$ is supercuspidal, then $\Theta(\pi)$ is either
zero or irreducible (and thus is equal to $\theta(\pi)$).  Moreover,
for any irreducible supercuspidal $\pi$ and $\pi'$,
\[ \Theta(\pi) \cong \Theta(\pi')\ne 0 \Longrightarrow \pi \cong \pi'.
\]

\vskip 5pt

\noindent (ii) $\theta(\pi)$ is multiplicity-free.

\vskip 5pt
\noindent (iii) If $p \ne 2$, the Howe duality conjecture holds.
\end{theorem} 

The statement (i) is a classic theorem of Kudla \cite{k83} (see also \cite{mvw}),
whereas (iii) is a well-known result of Waldspurger \cite{w90}. The
statement (ii), on the other hand, is a recent result of Li-Sun-Tian
\cite{lst}.  We note that the techniques for proving the three
statements in the theorems are quite disjoint from each other. For
example, the proof of (i) is based on arguments using the doubling
see-saw and Jacquet modules of the Weil representation: these have
become standard tools in the study of local theta correspondence. The
proof of (iii) is based on $K$-type analysis and uses various lattice
models of the Weil representation. Finally, the proof of (ii) is based
on an argument using the Gelfand-Kazhdan criterion for the
(non-)existence of equivariant distributions.  

\vskip 5pt

In this paper, we shall assume statements (i) and (ii), but not
statement (iii). Indeed, the purpose of this paper is to extend the
results of the above theorem to the case of more general $\pi$ and
arbitrary residue characteristic, using the same tools in the proof of
Theorem \ref{T:kudla-walds}(i).  More precisely, we shall prove the
following two results.  

\vskip 5pt

\begin{theorem}  \label{T:main}
 If $\pi$ is an irreducible tempered representation of $G(W)$, then $\theta(\pi)$ is either
zero or irreducible.  Moreover,
for any irreducible tempered $\pi$ and $\pi'$,
\[ \theta(\pi) \cong \theta(\pi')\ne 0 \Longrightarrow \pi \cong \pi'.
\]
When $m \leq (n+\epsilon_0) +1$, $\theta(\pi)$ is tempered if it is nonzero.
 \end{theorem}
 
\vskip 5pt

\begin{theorem} \label{T:equal} 
If $|m-( n+\epsilon_0)| \leq 1$, then the
Howe duality conjecture holds for $G(W) \times H(V)$.
\end{theorem}

\vskip 5pt

To be precise, Theorem \ref{T:equal} applies to
the dual pairs ${\rm O}_{2n+1} \times \Mp_{2n}$, ${\rm O}_{2n} \times
\Sp_{2n}$, ${\rm O}_{2n} \times \Sp_{2n+2}$, $\U_n \times \U_n$ and
$\U_n \times \U_{n+1}$.  We call them the (almost) equal rank dual
pairs. With Theorem \ref{T:equal}, the assumption
that $p \ne 2$ can be totally removed from all the results in
\cite{gs} (on the local Shimura correspondence) and also parts of
\cite{gi} (those dealing with the almost equal rank case, such as the
results on Prasad's conjecture in \cite[Appendix C]{gi}).  \vskip 5pt

\vskip 5pt

\vskip 5pt

We would like to point out some related results in the literature,
especially the paper \cite{rob} of Roberts and the papers \cite{M0,
muic, M1, M2} of Mui\'c: \vskip 5pt

\begin{enumerate}[$\bullet$]
\item In the context of Theorem \ref{T:main}, the temperedness of any
irreducible summand of $\theta(\pi)$ when $m \leq n+ \epsilon_0+1$ was checked by Roberts
\cite[Theorem 4.2]{rob}, at least for symplectic-even-orthogonal dual
pairs.  Further, the main idea in the proof of Theorem \ref{T:main}
can already be found in the proof of \cite[Theorem 4.4]{rob} (see
Proposition \ref{P:rob} below).  \vskip 5pt

\item In \cite{M0}, Mui\'c established Theorem \ref{T:main} (and much
more) for discrete series representations but the results there depended on the Moeglin-Tadi\'c (MT)
classification of discrete series representations, which was
conditional on some hypotheses.  We are not entirely sure whether the
MT classification is unconditional today. But our goal
here is to give a simple proof of the theorems above without resort to
classification.  \vskip 5pt

\item In \cite{M1}, Mui\'c dispensed with the MT classification and
proves some basic properties of $\Theta(\pi)$ for discrete series
representations $\pi$. For example, he showed in the context of Theorem
\ref{T:main} that $\Theta(\pi)$ is the direct sum of discrete series
representations if $m\leq n+\epsilon_0$. The techniques of proof used
in \cite{M1} are almost entirely based on the analysis of Jacquet
modules. The paper \cite{M1} does not establish Theorem \ref{T:main},
but we shall make use of some results such as \cite[Theorem 4.1, 4.2 and 6.1]{M1} in our
proof of Theorem \ref{T:main}.  In \cite[Prop. 8.1]{gs} and
\cite[Prop. C.1]{gi}, a self-contained and more streamlined proof of the relevant 
parts of \cite[Theorem 4.1, 4.2 and 6.1]{M1} was given for all the dual pairs
considered here.  We shall revisit and extend  this simpler proof in Proposition \ref{P:ds}.

\vskip 5pt

\item In \cite{M2}, assuming Theorem \ref{T:main} for discrete series representation, 
Mui\'c studied the
theta lifts of tempered representations. He determined $\Theta(\pi)$
(in terms of the theta lifts of discrete series representations) and showed the
irreducibility of $\theta(\pi)$ in many cases.  However, the proof of
parts of the main results \cite[Theorems 5.1 and 5.2]{M2} depended on
the MT classification in its use of \cite[Theorem 6.6]{M2}. We do not
use results from \cite{M2} in this paper. Rather, our Theorem
\ref{T:main} renders most of \cite{M2} unconditional,
and completes some results there.
\end{enumerate} 

\vskip 5pt

The proofs of the results in \cite{M0,muic, M1,M2} are based on some
intricate and explicit computations of Jacquet modules and some
detailed knowledge (short of classification) of the discrete series
representations of classical groups. On the other hand, the results of
this paper are proved in a simpler and more conceptual manner, with
the more intricate computations already done in \cite{gs, gi, rob}.
It amounts to an attempt to prove the Howe duality conjecture using
the techniques and principles found in \cite{k83} and \cite{kr05},
supplemented by \cite{lst}, so as to remove the $p\ne 2$ assumption in
Waldspurger's theorem.   

\vskip 5pt

 Shortly after the completion of this paper, the authors have succeeded in giving the proof of the full Howe duality conjecture. 
 The proof builds upon the techniques of this paper and will appear in \cite{gt}.  While the proof given in \cite{gt} uses the doubling see-saw argument of \S 2, the 
argument in \S 3 is completely replaced by a different  idea which
originated in Minguez's thesis \cite{mi}, with the result that the
tempered representations do not play any special role in \cite{gt}. In
view of this, we have decided to keep this paper as is, even though the results here are subsumed by \cite{gt}, especially since some useful results about the theta lifts of tempered representations are shown here (such as Proposition \ref{P:ds}).  
\vskip 5pt

\noindent{\bf Acknowledgments:} This project was begun and to a large
extent completed during the authors' participation in the Oberwolfach
workshop ``Modular Forms" during the week of April 28, 2014. We thank
the organizers J. Bruinier, A. Ichino, T. Ikeda and O. Imamoglou for
their invitation to the workshop and the Oberwolfach Institute for the
excellent and peaceful working conditions.  We especially want to
thank Goran Mui\'c for answering various questions about his papers
\cite{M0,muic, M1, M2} and  Alberto Minguez for 
enlightening discussions which lead us to complete the proof of the
Howe duality conjecture in \cite{gt}. 

\vskip 5pt

This paper is dedicated to Jim Cogdell who has inspired us not only
through his mathematics but also through his lucid expositions of often technical subjects in the theory of
automorphic forms, his warm personality  and his attitude of service
which  is evident to all, especially  in his interaction with his late
mentor, Professor Piatetski-Shapiro. Jim's work at the beginning of
his career concerns the construction of modular forms by theta
correspondence and we hope this paper is an appropriate  contribution
to his 60th birthday volume.  

\vskip 10pt

\section{\bf Special Case of Theorem \ref{T:main}} 

Before beginning the proof of Theorem \ref{T:main}, let us specify the extra
data needed to consider the Weil representation of $G(W) \times H(V)$;
these are needed to split the metaplectic cover over the dual pair.
We shall follow the setup of \cite[\S 3.2-3.3]{gi} in fixing a pair of
splitting characters ${\bf \chi} = (\chi_V, \chi_W)$, which are
certain unitary characters of $E^{\times}$, with associated Weil
representation $\omega_{W, V, {\bf \chi}, \psi}$.  We shall frequently
suppress ${\bf \chi}$ and $\psi$ from the notation.  \vskip 5pt
\vskip 5pt

In this section, we shall first prove Theorem
\ref{T:main} for the case $m \leq n+ \epsilon_0$ or $m
>2(n+\epsilon_0)$. Indeed, what we will prove is slightly stronger
than Theorem \ref{T:main}, which is stated as follows:
\vskip 5pt

\begin{theorem} \label{T:temp} 
If $m \leq n+\epsilon_0$ or $m >2(n+\epsilon_0)$, then for any
tempered $\pi \in {\rm Irr}(G(W))$, $\theta(\pi)$ is either zero or irreducible.  Moreover, for tempered $\pi$ and any irreducible representation
$\pi'$,
\[ 0 \ne \theta(\pi) \subset \theta(\pi') \Longrightarrow \pi \cong
\pi'.
\]
When $m \leq  n+\epsilon_0$, $\theta(\pi)$ is tempered if it is
nonzero. (Note that unlike Theorem \ref{T:main}, we do not
assume $\pi'$ is tempered.)
 \end{theorem}

 \vskip 5pt
 
We consider the following see-saw diagram
\[ 
\xymatrix{ G(W \oplus W^-)\ar@{-}[d]_{}\ar@{-}[dr]&H(V)\times
H(V)\ar@{-}[d]\\ G(W)\times G(W^-)\ar@{-}[ur]&H(V)^{\Delta}, }
\] 
where $W^{-}$ denotes the space obtained from $W$ by multiplying
the form by $-1$, so that $G(W^-) = G(W)$.  
Given an irreducible tempered representation $\pi$ and any irreducible
representation $\pi'$ of $G(W)$,  the see-saw identity
\cite[\S 6.1]{gi} gives:
\[ 
\Hom_{G(W) \times G(W)} ( \Theta_{V, W + W^-}(\chi_W), \pi' \otimes
\pi^{\vee}\chi_V) = \Hom_{H(V)^\Delta}( \Theta(\pi') \otimes
\Theta(\pi)^{MVW}, \CC), 
\] 
where MVW refers to the involution on the
set of smooth representations of $H(V)$ introduced in \cite{mvw}.
Here $\Theta_{V,W + W^-}(\chi_W)$ denotes the big theta lift of the
character $\chi_W$  of $H(V)$ to $G(W+ W^-)$. It is a result of
Rallis that
\begin{equation}\label{E:Rallis}
\Theta_{V, W + W^-}(\chi_W) \hookrightarrow {\rm Ind}_{P(\Delta
W)}^{G(W+W^-)} \chi_V |\det|^{s_{m,n}} 
\end{equation}
where
\begin{enumerate}[$\bullet$]
\item $\Delta W \subset W + W^-$ is diagonally embedded and is a
maximal isotropic subspace;
\item $P(\Delta W)$ is the maximal parabolic subgroup of $G(W+ W^-)$
which stabilizes $\Delta W$ and  has Levi factor $\GL(\Delta W)$;
\item ${\rm Ind}_{P(\Delta W)}^{G(W+W^-)} \chi_V |\det|^{s}$ denotes
the degenerate principal series representation induced from the
character $\chi_V|\det|^s$ of $P(\Delta W)$ (normalized
induction);
\item moreover,
\[ s_{m,n} = \frac{m-(n+\epsilon_0)}{2}. \]
 \end{enumerate} \vskip 5pt

We consider the two cases in turn.  \vskip 5pt

\noindent{\bf \underline{Case 1: $m \leq n+ \epsilon_0$}.}  \vskip 5pt

We  first note that one can prove the temperedness of  $\theta(\pi)$ 
(if nonzero) in the same way as in
\cite[Prop. C.1 and Prop. C.4(i)]{gi}. Hence, we will focus on the rest of the theorem.
\vskip 5pt

In this case, $s_{m,n} \leq 0$ and there is a surjective map (see
\cite[Prop. 8.2]{gi})
\begin{equation}\label{E:surjection}
{\rm Ind}_{P(\Delta W)}^{G(W+W^-)}\chi_V |\det|^{-s_{m,n}}
\longrightarrow \Theta_{V, W + W^-}(\chi_W). 
\end{equation}
Hence the see-saw
identity gives:
\[ \Hom_{G(W) \times G(W)}( {\rm Ind}_{P(\Delta W)}^{G(W+W^-)} \chi_V
|\det|^{-s_{m,n}}, \pi' \otimes \pi^{\vee}\chi_V) \supset \Hom_{H(V)}
(\theta(\pi') , \theta(\pi)).  \] 
To prove the theorem, it suffices to
show that the LHS has dimension $\leq 1$, with equality only if $\pi =
\pi'$.  \vskip 5pt

For this, we need the following crucial lemma (see \cite{kr05}), which
we shall state in slightly greater generality here for later use.  Let
$W' = W + \mathbb{H}^r$ where $\mathbb{H}$ is the hyperbolic plane
(i.e.\ the split $-\epsilon$-Hermitian space of dimension 2), so
that $n' := \dim W' = n + 2r$.  Consider the split space $\mathbb{W} =
W' + W^-$.  For a maximal isotropic subspace $Y$ of $\mathbb{H}^r$,
the space $\Delta W \oplus Y$ is a maximal isotropic subspace of
$\mathbb{W}$, whose stabilizer in $G(\mathbb{W})$ is a maximal
parabolic subgroup $P$. Now we have: \vskip 5pt

\begin{lemma} \label{L:key} 
As a representation of $G(W') \times
G(W^-)$, $ {\rm Ind}_{P}^{G(\mathbb{W})} \chi_V \cdot |\det|^{s}$
possesses an equivariant filtration
\[ 0 \subset I_0 \subset I_1 \subset\cdots\subset I_q = {\rm
Ind}_{P}^{G(\mathbb{W})} \chi_V \cdot |\det|^{s} \] with successive
quotients
\begin{align}
 R_t &= I_t / I_{t-1}   \notag \\
 &= {\rm Ind}_{Q'_{t+r} \times Q_t}^{G(W') \times
G(W^-)} \Big( \big(\chi_V |{\det}_{X_{t+r}}|^{s + \frac{t}{2}} \boxtimes
\chi_V |{\det}_{X_t}|^{s+\frac{t}{2}} \big) \otimes\\
&\hspace{2in} (\chi_V \circ
{\det}_{W^-_{n-2t}})\otimes  C^{\infty}_c( G(W_{n-2t})) \Big). \notag 
\end{align} 
Here, the induction is normalized and
\begin{enumerate}[$\bullet$]
\item $q$ is the Witt index of $W$;
\item $Q_t$ is the maximal parabolic subgroup of $G(W)$ stabilizing a
$t$-dimensional isotropic subspace $X_t$ of $W$, with Levi subgroup
$\GL(X_t) \times G(W_{n-2t})$, where $\dim W_{n-2t} = n-2t$.
\item $Q'_{t+r}$ is the maximal parabolic subgroup stabilizing the
$(t+r)$-dimensional isotropic subspace $X_t + Y$ of $W'$ with Levi
factor $\GL ( X_t + Y) \times G(W_{n-2t})$.
\item $G(W_{n-2t})\times G(W_{n-2t})$ acts on $C^{\infty}_c(
G(W_{n-2t}))$ by  left-right translation.
\end{enumerate} 
In particular,
\[ R_0 =  {\rm Ind}_{Q' (Y) \times G(W^-)}^{G(W') \times G(W^-)} \chi_V
|{\det}_Y|^s \otimes (\chi_V \circ {\det}_{W^-}) \otimes C^{\infty}_c(G(W)). \]
\end{lemma} 

\vskip 5pt 

We shall apply this lemma with $W' = W$, in which
case $R_0 = (\chi_V \circ {\det}_{W^-}) \otimes C^{\infty}_c(G(W))$.
Then we claim that the natural restriction map
\begin{align*} 
&\Hom_{G(W) \times G(W)}( {\rm Ind}_{P(\Delta W)}^{G(W+W^-)}
|\det|^{-s_{m,n}}, \pi' \otimes \pi^{\vee}\chi_V ) \\
&\hspace{2in}\longrightarrow \Hom_{G(W)
\times G(W)}( R_0, \pi' \otimes \pi^{\vee}\chi_V)
\end{align*}
 is injective. This will
imply the theorem since the RHS has dimension $\leq 1$, with equality
if and only if $\pi = \pi'$.  \vskip 5pt

To deduce the claim, it suffices to to show that for each $0 < t \leq q$,
\[ \Hom_{G(W) \times G(W)}( R_t , \pi' \otimes \pi^{\vee}\chi_V) = 0. \] 
By
Frobenius reciprocity, $\Hom_{G(W) \times G(W)}( R_t , \pi' \otimes
\pi^{\vee}\chi_V)$ is equal to
\begin{align*} 
&\Hom_{L(X_t) \times L(X_t)} \Big( \big( \chi_V |\det|^{-s_{m,n} +
\frac{t}{2}} \boxtimes\chi_V |\det|^{-s_{m,n}+\frac{t}{2}} \big)
\otimes C^{\infty}_c( G(W_{n-2t})),\\
&\hspace{150pt} R_{\overline{Q}_t}(\pi') \otimes R_{\overline{Q}_t} (\pi^{\vee})\Big)
\end{align*}
where
 \[ 
L(X_t) = \GL(X_t) \times G(W_{n-2t})
\]
is the Levi factor of
$Q_t$.  Here and elsewhere $R_{\overline{Q}_t}$ indicates the
normalized Jacquet functor with respect to $\overline{Q}_t$. But $-s_{m,n} +
\frac{t}{2} > 0$ whereas, since $\pi$ is
tempered, it follows by Casselman's criterion that the center of
$\GL_t$ acts on any irreducible subquotient of $R_{\overline{Q}_t}
(\pi^{\vee})$ by a character of the form $\mu \cdot |-|^{\alpha}$ with
$\mu$ unitary and $\alpha \leq 0$.  Hence we deduce that the above Hom
space is $0$, as desired, and we have proved Theorem \ref{T:temp} when
$m \leq n+ \epsilon_0$.  

\vskip 10pt

\noindent{\bf \underline{Case 2: $m > 2(n+ \epsilon_0)$}.}

\vskip 5pt

In this case, $s_{m,n} >0$ is so large that the degenerate principal
series representation ${\rm Ind}_{P(\Delta W)}^{G(W+ W^-)} \chi_V
\cdot |\det|^{s_{m,n}}$ is irreducible \cite[Proposition 7.1]{gi}, and hence
\[ 
\Theta_{V, W+W^-}(\chi_W) = {\rm Ind}_{P(\Delta W)}^{G(W+ W^-)}
\chi_V|\det|^{s_{m,n}}. 
\] 
The same argument as in Case 1
completes the proof of Theorem \ref{T:temp}.  

\vskip 10pt

\vskip 10pt

\section{\bf  Proof of Theorem \ref{T:main}}

In this section, we complete the proof of
Theorem \ref{T:main}. In view of Theorem \ref{T:temp}, it remains to consider the  case $n+ \epsilon_0 < m
\leq 2(n+\epsilon_0)$.  For this case, we consider the theta lift of $\pi$ to the Witt tower $\{
V_{m'} \}$ of spaces containing $V = V_m$.  The following proposition
is the key technical result that we use  (see \cite[Theorems 4.1 and
4.2]{M1} and  \cite[Theorem 4.2]{rob}): 
\vskip 5pt

 \vskip 5pt

\begin{prop}  \label{P:ds}
Assume $m>n+ \epsilon_0$. For tempered $\pi$, any irreducible quotient
$\sigma$ of $\Theta_{W,  V_m}(\pi)$ is either tempered or is the
Langlands quotient of a standard module
\[  {\rm Ind}_P^{H(V_m)} \tau_1\cdot \chi_W | \det |^{s_1} \otimes\cdots\otimes
\tau_k \cdot \chi_W | \det |^{s_k} \otimes \sigma' \]
with  $\tau_i$ a unitary discrete series representation of some
$\GL_{n_i}$, $\sigma'$ a tempered representation of   $H(V_{m'})$
with $m' = m - 2\sum_i n_i$, and
\[  s_1 \geq s_2 \geq\cdots\geq s_k  > 0,\] 
satisfying:
\begin{itemize}
\item[(a)]  $\tau_1= {\bf 1}$ (so $n_1=1$) and $s_1  = \frac{m- (n+\epsilon_0) -1}{2}$, or
\item[(b)]  $\tau_1 = {\rm St}_{n_1}$ and $s_1  =
  \frac{1}{2}$, where ${\rm St}_{n_1}$ denotes  the Steinberg
  representation of $\GL_{n_1}$ with $n_1 = m- (n+\epsilon_0)-1  > 1$.
\end{itemize}
Indeed, (a)  could hold only if $m > n + \epsilon_0 +1$ and
$\Theta_{W, V_{m-2}}(\pi) \ne 0$, and (b) could hold only if  $m >
n+\epsilon_0 +2$ and the square-integrable support of $\pi$  contains
a (twisted) Steinberg representation $\chi_V \cdot {\rm St}_{n_1-1}$.
\end{prop}

 \vskip 5pt

\noindent  Here, by the square-integrable support of a tempered
representation $\pi$, we mean the (unique up to association) set $\{
\tau_1,\dots,\tau_r, \pi'\}$
of essentially square-integrable representations  such that
$\pi$ is contained in the representation parabolically induced from
the representation $\tau_1 \boxtimes \cdots\boxtimes \tau_r \boxtimes \pi'$ of a Levi
subgroup $\GL_{n_1} \times\cdots\times \GL_{n_r}  \times G(W')$ of $G(W)$.  
\vskip 5pt

So as not to disrupt the proof of Theorem \ref{T:main}, we postpone
the proof of this proposition to the last section.
We also need the following refinement of a result of Roberts \cite[Theorem 4.4]{rob};
we include the proof so as to cover all the dual pairs considered
here.  \vskip 5pt

\begin{prop} \label{P:rob}
Let $\pi$ be an irreducible representation of $G(W)$ and let $V_{d'}
=  V_d \oplus \mathbb{H}^r$, where $\mathbb{H}$ is the split
$\epsilon$-Hermitian space of  dimension 2, so that $d'  \geq  d \geq
n+\epsilon_0$. Suppose that
\[   \sigma  \subset \theta_{W, V_d}(\pi)  \quad \text{and} \quad
\sigma'  \subset \theta_{W, V_{d'}}(\pi)\]
are irreducible representations. 
 If $d'  = d +2r$ is sufficiently large, then $\sigma'$ is a quotient of the representation
\[  
\Ind_Q^{H(V_{d'})}(\chi_W |-|^{\frac{d'   - (n+\epsilon_0)-1}{2}}\otimes\chi_W
|-|^{\frac{d'-(n+\epsilon_0) -3}{2}}\otimes\cdots\otimes\chi_W
|-|^{\frac{d - (n+\epsilon_0)+1}{2}}\otimes\sigma) \]
induced from the parabolic subgroup $Q$ with Levi factor $(\GL_1)^{r}  \times H(V_d)$.
\vskip 5pt

Furthermore, if $\sigma$ is tempered, then the above conclusion holds
for all $ r\geq 0$, in which case the above induced representation is a standard
module and $\sigma'$ is its unique Langlands quotient. 
\end{prop}

\begin{proof}
 Consider the see-saw diagram
\[
    \xymatrix{
    H(V_{d'} \oplus V_d^-)\ar@{-}[d]_{i}\ar@{-}[dr]&G(W)\times G(W)\ar@{-}[d]\\
    H(V_{d'})\times H(V_d^-)\ar@{-}[ur]&G(W)^{\Delta},
    }
\]
and let
\[  s_{d+r,n}  = \frac{(d+r)  - (n+ \epsilon_0)}{2}  >  0. \]
The see-saw identity  gives:
\begin{align}
0 &\ne \Hom_{G(W)^\Delta}( \Theta(\sigma') \otimes
\Theta(\sigma)^{MVW},\mathbb{C}) \notag \\ 
&= \Hom_{H(V_{d'}) \times H(V_d^-)}(\Theta_{W, V_{d'}+V_d^-}(\chi_V),
\sigma' \otimes \sigma^{\vee}\chi_W) \notag \\ 
&\subseteq   \Hom_{H(V_{d'}) \times H(V_d)}( {\rm Ind}^{H(V_{d'} +
  V_d)}_{P(\Delta V_d + Y)} \chi_W |\det|^{s_{d+r,n}},  \sigma'
\otimes \sigma^{\vee}\chi_W),  \notag 
\end{align}
where we used the analogue of (\ref{E:surjection}) for the last inclusion. Here, note  that $s_{d+r,n}$
is the analogue of $-s_{m,n}$ in \eqref{E:surjection} and 
$Y$ is a maximal isotropic subspace of $\mathbb{H}^r$ so that $\Delta
V_d+Y$ is a maximal isotropic subspace of $V_{d'}+V_d$.
\vskip 5pt

Now we apply Lemma \ref{L:key} (or rather its analogue with the roles of
$W$ and $V$ exchanged), which describes an $H(V_{d'}) \times
H(V_d)$-equivariant filtration of  the induced representation ${\rm Ind}^{H(V_{d'} +
  V_d)}_{P(\Delta V_d + Y)} \chi_W |\det|^{s_{d+r,n}}$. Note that the
length of this filtration depends only on $V_d$ and not on $V_{d'}$.
When $r$ is sufficiently large, all the characters
$\chi_W|\det|^{s_{d+r, n} +\frac{t}{2} }$ which occur in the description of the
successive quotients $R_t$ of this filtration with $t >0$  will be
different from any central exponents of any Jacquet module of $\sigma$
(which is a finite set).  Indeed, this holds for all $r\geq 0$ when
$\sigma$ is tempered, as in the proof of Case 1 of Theorem
\ref{T:temp}.
\vskip5pt

Thus we see that when $r$ is sufficiently large,
\begin{align}
 0 \ne&\Hom_{H(V_{d'}) \times H(V_d)}( {\rm Ind}^{H(V_{d'} +
   V_d)}_{P(\Delta V_d + Y)} \chi_W |\det|^{s_{d+r,n}},  \sigma'
 \otimes \sigma^{\vee}\chi_W)  \notag  \\
 \subseteq 
 &\Hom_{H(V_{d'}) \times H(V_d)} \left(   R_0,  \sigma' \otimes \sigma^{\vee} \chi_W \right) \notag \\
=  &\Hom_{H(V_{d'}) \times H(V_d)} \left(  {\rm Ind}_{Q'(Y) \times H(V_d)}^{H(V_{d'}) \times H(V_d)}
 \chi_W  |{\det}_Y|^{s_{d+r, n}}  \otimes   C^{\infty}_c(H(V_d)),
 \sigma' \otimes \sigma^{\vee} \right) \notag \\
 = &\Hom_{\GL(Y)  \times H(V_d)  \times H(V_d)} \left( \chi_W |{\det}_Y|^{s_{d+r,n}}  \otimes  
 C^{\infty}_c(H(V_d)),  R_{\overline{Q}'(Y)} (\sigma')  \otimes \sigma^{\vee} \right) \notag \\
  = &\Hom_{ \GL(Y) \times H(V_d)} \left(\chi_W  |{\det}_Y|^{s_{d+r,
        n}}  \otimes \sigma , R_{\overline{Q}'(Y)} (\sigma') \right)
  \notag \\
 =  &\Hom_{H(V_{d'})}\left(  {\rm Ind}_{Q'(Y)}^{H(V_{d'})} (\chi_W
   |{\det}_Y|^{s_{d+r, n}} \otimes \sigma)   ,  \sigma'
 \right). \notag 
  \end{align}
  Thus, $\sigma'$ is a quotient of ${\rm Ind}_{Q'(Y)}^{H(V_{d'})}
  (\chi_W |{\det}_Y|^{s_{d+r, n}} \otimes \sigma)$. But the latter is
  a quotient of the induced representation given in the proposition.  
 \vskip 10pt
 
 When $\sigma$ is tempered,  the above conclusions hold for any $r\geq 0$
 and it is clear that the induced representation is a standard
 module. This completes the proof of the proposition.  
 
 \end{proof}

 We can now complete the proof of Theorem \ref{T:main}.
 Suppose that $\sigma_1$ and $\sigma_2$ are both irreducible summands
 of $\theta_{W,V_m}(\pi)$, with $m >  n+ \epsilon_0$.
 In the context of Proposition \ref{P:rob}, we take $d = m$ and $d' =
 m + 2r$ sufficiently large.  Let $\sigma'  = \theta_{W, V_{d'}}(\pi)$
 (which is irreducible for $d'$ sufficiently large by Theorem
 \ref{T:temp}).  By Proposition \ref{P:rob}, we conclude that
 $\sigma'$ is a quotient of 
 \[ \Sigma_i =  \Ind_Q^{H(V_{d'})}(\chi_W |-|^{\frac{d'   - (n+\epsilon_0)-1}{2}}\otimes
 \chi_W |-|^{\frac{d'-(n+\epsilon_0) -3}{2}}\otimes\cdots\otimes\chi_W |-|^{\frac{d
     - (n+\epsilon_0)+1}{2}}\otimes\sigma_i) \]
 for $i =1$ or $2$. 
 Now we claim that $\Sigma_i$  is a standard module. This is clear if
 $\sigma_i$ is tempered.  On the other hand, if $\sigma_i$ is
 nontempered, then Proposition \ref{P:ds} describes two possibilities
 (a) and (b) for $\sigma_i$. In either case, $\sigma_i$ is the
 Langlands quotient of a standard module 
 \[  {\rm Ind}_P^{H(V_m)}( \tau_1 \cdot \chi_W | \det |^{s_1} \otimes\cdots\otimes
 \tau_k \cdot \chi_W | \det |^{s_k} \otimes \sigma_0) \]
 with
 \[  0 < s_1  \leq \frac{ d - (n+\epsilon_0) -1}{2}. \]
 It follows from this that $\Sigma_i$ is a standard module and
 $\sigma'$ is its unique Langlands quotient.
 By the uniqueness of Langlands quotient data, we must have $\sigma_1 \cong \sigma_2$.
 Hence, we conclude that $\theta_{W, V_m}(\pi)$ is
isotypic, and  it follows from Theorem \ref{T:kudla-walds}(ii) that
$\theta_{W, V}(\pi)$ is irreducible for tempered $\pi$.

\vskip 5pt

Finally, suppose that $\theta(\pi_1) \cong \theta(\pi_2) \cong\sigma \ne 0$
for two tempered representations $\pi_1$ and $\pi_2$. Since we are assuming
$n+\epsilon_0+1 \leq m \leq 2(n+\epsilon_0)$, the possibilities for
$\sigma$ are given in Proposition \ref{P:ds}. 
Now take $d' = m +2r$ sufficiently large in Proposition \ref{P:rob}.
If $\sigma_1=\theta_{W, V_{d'}}(\pi_1)$ and $\sigma_2= \theta_{W,
  V_{d'}}(\pi_2)$, then both $\sigma_1$ and $\sigma_2$ will be the
Langlands quotient of  the same standard module 
 \[  \Ind_Q^{H(V_{d'})}(\chi_W |-|^{\frac{d'   -
     (n+\epsilon_0)-1}{2}}\otimes\chi_W |-|^{\frac{d'-(n+\epsilon_0)
     -3}{2}}\otimes\cdots\otimes\chi_W |-|^{\frac{m - (n+\epsilon_0)+1}{2}}
\otimes \sigma),\]
where $Q$ is as in Proposition \ref{P:rob}. This implies that
$\sigma_1\cong \sigma_2$. By Theorem \ref{T:temp}, we deduce that
$\pi_1\cong \pi_2$.

\vskip 10pt

This completes the proof of Theorem \ref{T:main}.

  \vskip 10pt
\section{\bf Proof of Theorem \ref{T:equal}} 

We shall now show Theorem \ref{T:equal} and without loss of
generality, we may assume
that $0 \leq m-(n+\epsilon_0) \leq 1$. By Theorem \ref{T:main}, we already know that
if $\pi$ is tempered, then $\theta(\pi)$ is irreducible tempered or $0$.
Thus it remains to treat the nontempered case. For this, we need the
following lemma which gives more precise control on the big theta lift
of tempered representations.  \vskip 5pt
 
 \begin{lemma}   \label{L:equal}
 Assume that $0 \leq m-(n+\epsilon_0) \leq 1$. If $\pi$ is
tempered, then $\Theta(\pi) = \theta(\pi)$, so that $\Theta(\pi)$ is
irreducible tempered or $0$.
\end{lemma}

\vskip 5pt

\begin{proof} This was shown in \cite[Prop. 8.1(i) and (ii)]{gs} and
\cite[Prop. C.1 and Prop. C.4(i)]{gi}.
 \end{proof} 

\vskip 5pt

Now if $\pi$ is nontempered, it can be expressed uniquely as the
Langlands quotient of a standard module 
\begin{equation}\label{E:Lang_quotient}
  {\rm Ind}^{G(W)}_{P_{r_1,\dots,r_k}}  \tau_1\cdot \chi_V  |\det|^{s_1}  \otimes
\tau_2 \cdot \chi_V  |\det|^{s_2}  \otimes\cdots\otimes \tau_k \cdot
\chi_V  |\det|^{s_k}  \otimes\pi_0
\end{equation}
of $G(W)$, where $P_{r_1,\dots,r_k}$ is a parabolic subgroup of $G(W)$
whose Levi factor is $\GL_{r_1} \times\cdots\times \GL_{r_k}  \times
G(W')$, so that 
$\dim W'  =  n - 2\sum_i r_i$, $\pi_0$ is a tempered
representation of $G(W')$, each $\tau_i$ is a unitary tempered
representation of $\GL_{r_i}$ and $s_1 >\cdots> s_k  >0$.   In
\cite[Prop. C.4(ii)]{gi} and \cite[Prop. 8.1(iii)]{gs}, it was shown
that $\Theta(\pi)$ is a quotient of the induced representation
\begin{equation}\label{E:Lang_quotient2}
   {\rm Ind}^{H(V)}_{Q_{r_1,\dots,r_k}}  \tau_1 \cdot \chi_W |\det|^{s_1}  \otimes
\tau_2\cdot \chi_W  |\det|^{s_2}  \otimes\cdots\otimes \tau_k \cdot \chi_W |\det|^{s_k}  \otimes
\Theta_{W', V'}(\pi_0)
\end{equation}
of $H(V)$, where $Q_{r_1,\dots,r_k}$ is the parabolic subgroup of $H(V)$
whose Levi factor is $\GL_{r_1} \times\cdots\times \GL_{r_k}  \times
H(V')$, so that  $\dim V'  =  m - 2\sum_i r_i$. Since $0 \leq \dim V'
- (\dim W' + \epsilon_0) \leq 1$, Lemma \ref{L:equal} implies that
$\Theta_{W', V'}(\pi_0)$ is an irreducible tempered representation, so
that the above induced representation is a standard module of $H(V)$.
In particular, $\theta(\pi)$ is either $0$ or is the unique
Langlands quotient of that standard module.  \vskip 5pt

Finally, assume that $\theta(\pi)\cong\theta(\pi')\neq 0$.  Express  $\pi$ 
as the Langlands quotient of a standard module as in
\eqref{E:Lang_quotient}, though this time we allow the case
$P_{r_1,\dots,r_k}=G(W)$ (in which case $\pi$ is tempered). Similarly express  $\pi'$ as a Langlands quotient, possibly with a
different quotient data. Then by the above argument, $\theta(\pi)$ is the Langlands quotient
of the induced representation \eqref{E:Lang_quotient2}, and similarly
$\theta(\pi')$ is the Langlands quotient of the analogous induced
representation. By the uniqueness of the Langlands quotient data
and Theorem \ref{T:main}, we deduce that $\pi\cong\pi'$.

\vskip 5pt
This completes the proof of Theorem \ref{T:equal} which establishes
the Howe duality conjecture in the (almost) equal rank case.

\vskip 10pt

\section{\bf Proof of Proposition \ref{P:ds}} 

In this section, we give the proof of Proposition \ref{P:ds} following
that of \cite[Prop. C.1]{gi}.

 \vskip 5pt
\begin{enumerate}[$\bullet$]
 
\item Suppose that $\sigma$ is a nontempered irreducible quotient of
  $\Theta_{W, V_{m}}(\pi)$. Suppose that $\sigma$ is the Langlands
  quotient of a standard module
\begin{equation}  \label{E:standard}
  {\rm Ind}_P^{H(V_m)} \tau_1\cdot \chi_W  | \det |^{s_1} \otimes\cdots\otimes
  \tau_k \cdot \chi_W | \det |^{s_k} \otimes \sigma' \end{equation}
with  $\tau_i$ unitary discrete series representations of some
$\GL_{n_i}$, $\sigma'$ a tempered representation of  some $H(V_{m'})$
with $m' < m$, and
\[  s_1 \geq s_2 \geq\cdots\geq s_k  > 0.\] 
We need to show that only possibilities (a) and (b) as given in
Proposition \ref{P:ds} can occur.  

\vskip 5pt

\item  Let $t=n_1$. From the standard module above, we see that there exists a maximal
parabolic subgroup $Q=Q(Y_t)$ of $H= H(V_m)$ stabilizing a
$t$-dimensional isotropic subspace $Y_t$, with Levi component $L(Y_t)
= \GL(Y_t) \times H(V_{m-2t})$, such that
\[ \sigma \hookrightarrow \Ind^{H}_{Q}(\tau \cdot \chi_W |\det|^{-s_1}
\otimes \sigma_0).  \] 
Here, we have written $V = Y_t \oplus V_{m-2t}
\oplus Y_t^*$ with $Y_t^*$ isotropic, $\tau=(\tau_1^c)^\vee$, where
$^c$ indicates the conjugation by the generator of $\operatorname{Gal}(E/F)$, $s_1 >
0$ is the leading exponent as in (\ref{E:standard}) and $\sigma_0$ is
an irreducible representation of $H(V_{m-2t})$.
Thus we have  a nonzero $G(W) \times H$-equivariant map
\[ \omega_{V_{m},W} \longrightarrow \pi \otimes
 {\rm Ind}^{H}_{Q}(
\tau\cdot \chi_W |\det|^{-s_1} \otimes \sigma_0). \] 
By Frobenius reciprocity, we have 
\[ \pi^{\vee} \hookrightarrow \Hom_{L(Y_t)}(R_Q(\omega_{V_{m},W} ),
\tau\cdot \chi_W|\det|^{-s_1} \otimes \sigma_0). \] 
\vskip 5pt

\item By \cite{k83}, the Jacquet module $R_Q(\omega_{V_{m},W} )$ has
an equivariant filtration
\[ R_Q(\omega_{V_{m},W}) = R^0 \supset R^1 \supset \cdots \supset R^t
\supset R^{t+1} = 0 \] 
whose successive quotient $J^a = R^a/R^{a+1}$
is described in \cite[Lemma C.2]{gi}. 
More precisely,
\begin{align*}
 J^a& = \Ind^{\GL(Y_t) \times H(V_{m-2t}) \times G(W)}_{Q(Y_{t-a}, Y_t)
\times H(V_{m-2t}) \times P(X_a)}\\
&\qquad\left(\chi_W
|{\det}_{Y_{t-a}}|^{\lambda_{t-a}} \otimes C_c^\infty(\Isom_{E,c}(X_a,Y_a)) \otimes
\omega_{V_{m-2t}, W_{n-2a}}\right),
\end{align*}
where
\begin{enumerate}[\textperiodcentered]
\item $\lambda_{t-a} = ( n-m + t-a +\epsilon_0) /2 $,
\item $W = X_a + W_{n-2a} +X_a^*$ with $X_a$ an $a$-dimensional
isotropic space and $\dim W_{n-2a} = n-2a$,
\item $Y_t = Y_{t-a} + Y_a'$ and $Q(Y_{t-a}, Y_a)$ is the maximal
parabolic subgroup of $\GL(Y_t)$ stabilizing $Y_{t-a}$.
\item[\textperiodcentered] $\Isom_{E,c}(X_a, Y_a)$ is the set of
  $E$-conjugate-linear isomorphisms from $X_a$ to $Y_a$;
\item[\textperiodcentered] $\GL(X_a)\times\GL(Y_a)$ acts on
  $C_c^\infty(\Isom_{E,c}(X_a,Y_a))$ as 
  \[ ((b,c)\cdot f)(g)=\chi_V(\det b)\chi_W(\det
  c)f(c^{-1}g b) \]
   for $(b,c)\in \GL(X_a)\times\GL(Y_a)$, $f\in
  C_c^\infty(\Isom_{E,c}(X_a,Y_a))$ and $g\in\GL_a$.
\item $J^a=0$ for $a>\min\{t, q\}$, where $q$ is
  the Witt index of $W$.
\end{enumerate} 
In particular, the bottom piece of the filtration (if nonzero) is:
\[ J^t \cong \Ind^{\GL(Y_t) \times H(V_{m-2t}) \times G(W)}_{\GL(Y_t)
\times H(V_{m-2t}) \times P(X_t)} (C_c^\infty(\Isom_{E,c}(X_t,Y_t))\otimes \omega_{V_{m-2t},
W_{n-2t}}). \]
  \vskip 5pt

Thus for some $0 \leq a \leq t$, there is a nonzero map
\[ \pi^{\vee} \longrightarrow \Hom_{L(Y_t)}( J^a, \tau \cdot \chi_W |\det|^{-s_1}
\otimes \sigma_0). \] 
 We now consider different possibilities in turn.
 \vskip 5pt
 
\item Consider first the case when $a = t$. Then
\begin{align*}
0\neq &\Hom_{L(Y_t)}(J^t,  \tau \cdot \chi_W|\det|^{-s_1} \otimes \sigma_0)\\
 &\qquad\qquad\qquad
= \left(  \Ind^{G(W)}_{P(X_t)} \tau^{\vee} \cdot \chi_V |\det|^{s_1} 
 \otimes \Theta_{V_{m-2t}, W_{n-2t}}(\sigma_0) \right)^*
\end{align*}
so that one has an equivariant map
\[      
\Ind^{G(W)}_{P(X_t)} \tau^{\vee} \cdot \chi_V |\det|^{s_1}
 \otimes \Theta_{V_{m-2t}, W_{n-2t}}(\sigma_0)   \longrightarrow \pi. \]
If this is nonzero,  then by Frobenius reciprocity and Cansselman's
criterion, one has a  contradiction to the temperedness of $\pi$ (since $s_1 > 0$).
\vskip 5pt

\vskip 5pt

\item Now suppose that $t=1$ and $a=0$. Then
\[   \Hom_{L(Y_1)}( J^0, \tau \cdot \chi_W |\det|^{-s_1} \otimes \sigma_0) = \Hom_{L(Y_1)}( \chi_W |-|^{\lambda_1} \otimes \omega_{V_{m-2}, W}, \tau \cdot \chi_W |- |^{-s_1} \otimes \sigma_0).
 \]
This Hom space is nonzero if and only if  
\[  \tau = {\bf 1},  \quad   s_1 =  - \lambda_1  = \frac{m- (n+\epsilon_0) -1}{2} \quad \text{and}  \quad \Theta_{V_{m-2}, W}(\sigma_0) \ne 0. \]
For $\pi^{\vee}$ to embed into this Hom space, we need $\pi$ to be a quotient of $\Theta_{V_{m-2}, W}(\sigma_0)$, or equivalently $\sigma_0$ is a quotient of $\Theta_{W, V_{m-2}}(\pi)$.  
This gives the possibility (a) of the proposition. 
 \vskip 5pt

\item The remaining case is $t > a$ and $t> 1$.  Note that $t - a \geq
  1$. In this case, the non-vanishing of
$\Hom_{L(Y_t)}( J^a, \tau\cdot\chi_W|\det|^{-s_1} \otimes \sigma_0)$ is equivalent to
\begin{align} \label{E:non0} 
&\Hom_{\GL(Y_{t-a}) \times \GL(Y'_a)\times H(V_{m-2t})} \Big(
\chi_W  |{\det}_{Y_{t-a}}|^{\lambda_{t-a}}\otimes C_c^\infty(\Isom_{E,c}(X_a,Y_a)) \\
\notag&\qquad\qquad\qquad
\otimes \omega_{V_{m-2t}, W_{n-2a}},
R_{\overline{Q(Y_{t-a}, Y_t)}}(\tau)\cdot \chi_W |{\det}_{Y_t}|^{-s_1}
\otimes \sigma_0 \Big)\ne 0.
\end{align} \vskip 5pt

\item Since $\tau$ is an irreducible (unitary) discrete series
representation of $\GL(Y_t)$, by results of Zelevinsky (see \cite[Pg. 105]{M1}) we have
\[ R_{\overline{Q(Y_{t-a}, Y_t)}}(\tau) = \delta_1 |\det|^{e_1} \otimes
\delta_2|\det|^{e_2} \] for some irreducible (unitary) discrete series
representations $\delta_1$ and $\delta_2$ of $\GL(Y_{t-a})$ and
$\GL(Y_a')$ respectively, and some $e_1, e_2 \in \mathbb{R}$ such that
\begin{equation}
\label{E:t1t2} e_1 < e_2 \quad \text{and} \quad e_1 \cdot (t-a) + e_2
\cdot a = 0.
\end{equation} In particular, we must have $e_1 \le 0$. Note that if
$a=0$, then $e_1 = 0$.

\vskip 5pt

\item Now, the center of $\GL(Y_{t-a})$ acts on
$R_{\overline{Q(Y_{t-a}, Y_t)}}(\tau)\cdot \chi_W
|{\det}_{Y_t}|^{-s_1}$ by the character $\omega_{\delta_1} \cdot \chi_W
|\det|^{e_1 -s_1}$, whereas $\GL(Y_{t-a})$ acts on
$\chi_W |\det|^{\lambda_{t-a}} \otimes C_c^\infty(\Isom_{E,c}(X_a,Y_a)) \otimes
\omega_{V_{m-2t}, W_{n-2a}}$ by the character
$\chi_W |\det|^{\lambda_{t-a}}$. Here $\omega_{\delta_1}$ is the central
character of $\delta_1$ which is a unitary character.  For
(\ref{E:non0}) to hold, we must have $t-a =1$ (so that $a>0$), $\delta_1$ equal to the
trivial character and
 \[ e_1 -s_1 = \lambda_1 = (n +\epsilon_0 +1-m)/2 \in \frac{1}{2}
\mathbb{Z}. \] 
This has a chance of holding because both $e_1 - s_1$ and $\lambda_1$  are $<0$.
\vskip 5pt

\item Moreover, by results of Zelevinsky (see \cite[Pg. 105]{M1}), we must have
\[ \tau = {\rm St}_t \] so that
\begin{equation} \label{E:t} e_1 = -(t-1)/2, \quad e_2 =1/2 \quad
\text{and} \quad \delta_2 = {\rm St}_{t-1}.  
\end{equation} 
Then we deduce that
\begin{equation} \label{E:supp}
 \pi^{\vee} \hookrightarrow \left( {\rm Ind}_{P(X_a)}^{G(W)}
{\rm St}_a \chi_V |\det|^{s_1-e_2} \otimes \Theta_{V_{m-2t},
W_{n-2a}}(\sigma_0)\right)^{\vee}.  \end{equation}
 Since $\pi$ is tempered, we have
\begin{equation}  \label{E:s1}
 s_1 \leq  e_2 = 1/2.  
 \end{equation}

\vskip 5pt

\item To summarize, we have shown that  $a>0$,
\[ s_1 =e_1 -\lambda_1 = e_1 + \frac{m_0 - (n+\epsilon_0) -1}{2} \in
\frac{1}{2} \mathbb{Z} \] and
\[ 0 < s_1\leq \frac{1}{2}. \] 
 Hence $s_ 1 = 1/2$ and $\tau = {\rm St}_t$. Together with (\ref{E:t}), we deduce that
 \[  t = m - (n+\epsilon_0) -1  > 1.\]
 This gives possibility (b) in Proposition \ref{P:ds}. Moreover,
 (\ref{E:supp}) shows that when (b) holds,  \[ {\rm
   Ind}_{P(X_a)}^{G(W)} \chi_V \cdot {\rm St}_a \otimes
 \Theta_{V_{m-2t}, W_{n-2a}}(\sigma_0) \twoheadrightarrow \pi,  \] 
with $a = t-1>0$.
Hence there is an irreducible subquotient $\pi_0$ of
$\Theta_{V_{m-2t}, W_{n-2a}}(\sigma_0)$  such that
\[ {\rm Ind}_{P(X_a)}^{G(W)} \chi_V \cdot {\rm St}_a \otimes \pi_0 \twoheadrightarrow \pi.  \]
Then it follows by Casselman's temperedness criterion that $\pi_0$ is itself tempered.  
Hence, $\chi_V  \cdot {\rm St}_{a}$ is contained in the
square-integrable support of $\pi$.
 
  \end{enumerate}
 \vskip 5pt
 
This completes the proof of Proposition \ref{P:ds}.

   \vskip 10pt

\bibliographystyle{amsalpha}

\begin{thebibliography}{99}
\bibitem[GI]{gi}
W.~T.~Gan and A.~Ichino,
\emph{Formal degrees and local theta correspondence},
Invent. Math. 195 (2014), no. 3, 509-672.

\bibitem[GS]{gs}
W.~T.~Gan and G.~Savin,
\emph{Representations of metaplectic groups I: epsilon dichotomy and local Langlands correspondence},
Compos. Math. \textbf{148} (2012), 1655--1694.

\bibitem[GT]{gt}
W.~T.~Gan and S. Takeda,
\emph{A proof of the Howe duality conjecture}, preprint (2014).


 \bibitem[H]{howe}
R. Howe,
\emph{Transcending classical invariant theory},
J. Amer. Math. Soc. \textbf{2} (1989), 535--552.

\bibitem[K]{k83}
S.~S.~Kudla,
\emph{On the local theta-correspondence},
Invent. Math. \textbf{83} (1986), 229--255.

 \bibitem[KR]{kr05}
S.~S.~Kudla and S.~Rallis,
\emph{On first occurrence in the local theta correspondence},
Automorphic representations, $L$-functions and applications:
 progress and prospects,
Ohio State Univ. Math. Res. Inst. Publ. \textbf{11},
de Gruyter, Berlin, 2005, pp.~273--308.

\bibitem[LST]{lst}
J.-S.~Li, B.~Sun, and Y.~Tian,
\emph{The multiplicity one conjecture for local theta correspondences},
Invent. Math. \textbf{184} (2011), 117--124.

\bibitem[Mi]{mi} A. Minguez, \emph{Correspondance de Howe explicite:
    paires duales de type II}, Ann. Sci. \'Ec. Norm. Sup\'er. \textbf{41} (2008), 717--741.


\bibitem[MVW]{mvw}
C.~M{\oe}glin, M.-F.~Vign{\'e}ras, and J.-L.~Waldspurger,
\emph{Correspondances de Howe sur un corps $p$-adique},
Lecture Notes in Mathematics \textbf{1291}, Springer-Verlag, Berlin, 1987.

\bibitem[M1]{M0}  
G. Mui\'c,
\emph{Howe correspondence for discrete series representations; the
  case of $({\rm Sp}(n),{\rm O}(V))$},  J. Reine Angew. Math. 567
(2004), 99-150.

\bibitem[M2]{muic}
G.~Mui\'c, \emph{On the structure of the full lift for the Howe correspondence
 of $(\mathrm{Sp}(n), \mathrm{O}(V))$ for rank-one reducibilities},
Canad. Math. Bull. \textbf{49} (2006), 578--591.

\bibitem[M3]{M1}
 G. Mui\'{c}, {\em On the structure of theta lifts of discrete series for dual pairs $({\rm Sp}(n),{\rm O}(V))$}, Israel J. Math. 164 (2008), 87--124

\bibitem[M4]{M2} 
G. Mui\'{c}, {\em Theta lifts of tempered representations for dual pairs $({\rm Sp}_{2n}, {\rm O}(V))$}, Canadian J. Math. 60 No. 6 (2008), 1306-1335.

\bibitem[R]{rob}  B. Roberts,  {\em Tempered representations and the theta correspondence},
Canadian Journal of Mathematics, 50, (1998), 1105-1118. 

\bibitem[W]{w90}
J.-L.~Waldspurger,
\emph{D\'emonstration d'une conjecture de dualit\'e de Howe
 dans le cas $p$-adique, $p\ne 2$},
Festschrift in honor of I.~I.~Piatetski-Shapiro on the occasion of
 his sixtieth birthday, Part I,
Israel Math. Conf. Proc. \textbf{2}, Weizmann, Jerusalem, 1990, pp.~267--324.



\end{thebibliography}

\end{document}